\documentclass[12pt]{amsart}
\usepackage[active]{srcltx}
\usepackage{calc,amssymb,amsthm,amsmath,amscd, eucal,ulem}
\usepackage{alltt}
\usepackage[left=1.35in,top=1.25in,right=1.35in,bottom=1.25in]{geometry}
\RequirePackage[dvipsnames,usenames]{color}
\usepackage{xcolor}		
\usepackage{ulem}	

\input{kmacros3.sty}
\input{xy}
\xyoption{all}
\usepackage{tikz}

\numberwithin{equation}{theorem}

\renewcommand{\m}{\mathfrak{m}}
\renewcommand{\n}{\mathfrak{n}}

\DeclareMathOperator{\Spa}{Spa}

\usepackage{fullpage}

\usepackage{setspace}
\usepackage{hyperref}

\usepackage{enumerate}

\usepackage{graphicx}

\usepackage[all,cmtip]{xy}
\usepackage{verbatim}

\theoremstyle{theorem}

\begin{document}
\title{Big Cohen-Macaulay algebras and the vanishing conjecture for maps of Tor in mixed characteristic}
\author{Raymond Heitmann}
\author{Linquan Ma}
\address{Department of Mathematics\\ University of Texas at Austin\\ Austin \\ TX 78712}
\email{heitmann@math.utexas.edu}
\address{Department of Mathematics\\ University of Utah\\ Salt Lake City\\ Utah 84112}
\email{lquanma@math.utah.edu}
\maketitle
\begin{abstract}
We prove a version of weakly functorial big Cohen-Macaulay algebras that suffices to establish Hochster-Huneke's vanishing conjecture for maps of Tor in mixed characteristic. As a corollary, we prove an analog of Boutot's theorem that direct summands of regular rings are pseudo-rational in mixed characteristic. Our proof uses perfectoid spaces and is inspired by the recent breakthroughs on the direct summand conjecture by Andr\'{e} and Bhatt.
\end{abstract}

\section{Introduction and Preliminaries}

In a recent breakthrough, Y. Andr\'{e} \cite{AndreDirectsummandconjecture} settled Hochster's direct summand conjecture which dates back to 1969.

\begin{theorem}[Andr\'{e}]
\label{theorem--direct summand}
Let $A\to R$ be a finite extension of Noetherian rings. If $A$ is regular, then the map is split as a map of $A$-modules.
\end{theorem}

This was previously only known for rings containing a field \cite{HochsterTopicsintheHomologicaltheoryofmodules} and for rings of dimension less than or equal to three \cite{HeitmannDirectsummandconjectureindimensionthree}; what is new and striking is the general mixed characteristic case. A simplified and shorter proof of Theorem \ref{theorem--direct summand} was later found by Bhatt \cite{BhattDirectsummandandDerivedvariant}. But Andr\'{e}'s argument \cite{AndreDirectsummandconjecture} also proved the stronger conjecture that balanced big Cohen-Macaulay algebras exist in mixed characteristic.\footnote{Andr\'{e} stated the existence of big Cohen-Macaulay algebras for complete local domains \cite[Th\'{e}or\`{e}me 0.7.1]{AndreDirectsummandconjecture}, but the general case follows by killing a minimal prime and taking the completion.}
Recall that $B$ is called a {\it balanced big Cohen-Macaulay algebra} for the local ring $(R,\m)$ if $\m B\neq B$ and every system of parameters for $R$ is a regular sequence on $B$. It is a conjecture of Hochster that such algebras exist in general and he proved this for rings that contain a field \cite{HochsterTopicsintheHomologicaltheoryofmodules} \cite{HochsterBigCMalgebrasandembeddabilityWittvectors}. Andr\'{e}'s solution in mixed characteristic depends on his deep result in \cite{AndrePerfectoidAbhyankarLemma} that gives a generalization of the almost purity theorem: the perfectoid Abhyankar lemma.

The purpose of this paper is to prove that weakly functorial balanced big Cohen-Macaulay algebras exist for {\it certain} surjective ring homomorphisms in mixed characteristic, a result that has many applications.

\begin{theorem}[=Theorem \ref{theorem--big CM}]
\label{theorem--main1}
Let $(R,\m,k)$ be a complete local domain with $k$ algebraically closed, and let $Q\subseteq R$ be a height one prime ideal. Suppose both $R$ and $R/Q$ have mixed characteristic. Then there exists a commutative diagram:
\[\xymatrix{
R \ar[r] \ar[d] & R/Q \ar[d] \\
B \ar[r] & C
}
\]
where $B$, $C$ are balanced big Cohen-Macaulay algebras for $R$ and $R/Q$ respectively.
\end{theorem}

Our method of proving Theorem \ref{theorem--main1} {\it avoids} the perfectoid Abhyankar lemma in \cite{AndrePerfectoidAbhyankarLemma} and thus is much shorter than Andr\'{e}'s argument. More importantly, the weakly functorial property we prove is {\it new}.\footnote{In fact, our version of the existence of weakly functorial big Cohen-Macaulay algebras does not even seem to follow from the perfectoid Abhyankar lemma \cite{AndrePerfectoidAbhyankarLemma}.} We should note that, in equal characteristic, the existence of weakly functorial balanced big Cohen-Macaulay algebras was known in general \cite{HochsterHunekeApplicationsofbigCMalgebras}. Nonetheless, the version we prove is strong enough to settle Hochster-Huneke's vanishing conjecture for maps of Tor \cite{HochsterHunekeApplicationsofbigCMalgebras} in mixed characteristic.

\begin{theorem}[=Theorem \ref{theorem--vanishing theorem for maps of Tor}]
\label{theorem--main2}
Let $A\to R\to S$ be maps of Noetherian rings such that $A\to S$ is a local homomorphism of mixed characteristic regular local rings and $R$ is a module-finite torsion-free extension of $A$. Then for all $A$-modules $M$, the map $\Tor_i^A(M, R)\to \Tor_i^A(M, S)$ vanishes for all $i\geq 1$.
\end{theorem}

As a consequence of Theorem \ref{theorem--main2}, we prove the following, which is the mixed characteristic analog of Boutot's theorem \cite{BoutotRationalsingularitiesandquotientsbyreductivegroups}.\footnote{This corollary can be also proved by combining \cite[Remarque 4.2.1]{AndreDirectsummandconjecture} and \cite[Theorem 1.2]{BhattDirectsummandandDerivedvariant} (and an extra small argument), see Remark \ref{remark--Boutot also follows from Andre and Bhatt}. However, to the best of our knowledge, the results of \cite{AndreDirectsummandconjecture}, \cite{AndrePerfectoidAbhyankarLemma} and \cite{BhattDirectsummandandDerivedvariant} are not enough to establish Theorem \ref{theorem--main2}.} 

\begin{corollary}[=Corollary \ref{corollary--mixed char Boutot}]
\label{corollary--main}
If $R\to S$ is a ring extension such that $S$ is regular and the map is split as a map of $R$-modules, then $R$ is pseudo-rational (in particular Cohen-Macaulay).
\end{corollary}

It is well known that Theorem \ref{theorem--main2} implies Theorem \ref{theorem--direct summand} (for example, see \cite{Ranganathanthesis} or \cite[Remark 4.6]{MaThevanishingconjectureformapsofTorandderivedsplinters}). Recently, Bhatt \cite{BhattDirectsummandandDerivedvariant} gave an alternative and shorter proof of Theorem \ref{theorem--direct summand}: instead of using the perfectoid Abhyankar lemma, Bhatt established a quantitative form of Scholze's Hebbarkeitssatz (the Riemann extension theorem) for perfectoid spaces, and the same idea leads to a proof of a derived variant, i.e., the derived direct summand conjecture. We point out that Theorem \ref{theorem--main2} formally implies such derived variant by \cite[Remark 5.12]{MaThevanishingconjectureformapsofTorandderivedsplinters} and hence we recover, and in fact generalize, Bhatt's result (see Remark \ref{remark--derived direct summand}). Furthermore, although the idea is inspired by \cite{BhattDirectsummandandDerivedvariant}, our argument is expositionally independent of \cite{BhattDirectsummandandDerivedvariant}. We avoid the use of Scholze's Hebbarkeitssatz and the vanishing theorems of perfectoid spaces, instead we study the colon ideals of $A_\infty\langle\frac{p^n}{g}\rangle$ in Lemma \ref{lemma-annihilator of almost projective module}.

\begin{remark}
We should point out that, to the best of our knowledge, Hochster-Huneke's vanishing conjecture for maps of Tor is still open if $A$ and $R$ have mixed characteristic but $S$ has equal characteristic $p>0$. This case also implies Theorem \ref{theorem--direct summand} by \cite[(4.4)]{HochsterHunekeApplicationsofbigCMalgebras}. However, the discussion above shows that the mixed characteristic case we proved (i.e., Theorem \ref{theorem--main2}) is enough for almost all applications.
\end{remark}


This paper is organized as follows. In Section 2 we demonstrate a weakly functorial construction of integral perfectoid algebras in Lemma \ref{lemma--commutative diagram at the almost level}. Then, in Section 3, we prove Theorem \ref{theorem--main1}, and in Section 4, we prove Theorem \ref{theorem--main2} and Corollary \ref{corollary--main}.

\subsection{Perfectoid algebras} We will freely use the language of perfectoid spaces \cite{ScholzePerfectoidspaces} and almost mathematics \cite{GabberRameroAlmostringtheory}. In this paper we will always work in the following situation: for a perfect field $k$ of characteristic $p>0$, we let $W(k)$ be the ring of Witt vectors with coefficients in $k$. Let $K^\circ$ be the $p$-adic completion of $W(k)[p^{1/p^\infty}]$ and $K=K^\circ[1/p]$. Then $K$ is a {\it perfectoid field} in the sense of \cite{ScholzePerfectoidspaces} with $K^\circ\subseteq K$ its ring of integers.

A {\it perfectoid $K$-algebra} is a Banach $K$-algebra $R$ such that the set of powerbounded elements $R^\circ\subseteq R$ is bounded and the Frobenius is surjective on $R^\circ/p$. A $K^\circ$-algebra $S$ is called {\it integral perfectoid} if it is $p$-adically complete, $p$-torsion free, satisfies $S=S_*$\footnote{$S_*=\{x\in S[1/p]| p^{1/p^k}\cdot x\in S \text{ for all } k \}$. Hence $S$ is almost isomorphic to $S_*$ with respect to $(p^{1/p^\infty})$, thus in practice we will often ignore this distinction since one can always pass to $S_*$ without affecting the issue.} and the Frobenius induces an isomorphism $S/p^{1/p}\to S/p$. These two categories are equivalent to each other \cite[Theorem 5.2]{ScholzePerfectoidspaces} via the functors $R\to R^\circ$ and $S\to S[1/p]$.

Unless otherwise stated, almost mathematics in this paper will always be measured with respect to the ideal $(p^{1/p^\infty})$ in $K^\circ$.

\subsection{Partial algebra modifications}We briefly recall Hochster's partial algebra modifications that plays a crucial rule in the construction of balanced big Cohen-Macaulay algebras. Our definition and usage of these modifications is basically the same as that in \cite[Sections 3 and 4]{HochsterBigCohen-Macaulayalgebrasindimensionthree}. 

Let $(R,\m)$ be a local ring and let $M$ be an $R$-module. We define a {\it partial algebra modification} of $M$ with respect to a system of parameters $x_1,\dots,x_d$ of $R$ to be a map $M\to M'$ obtained as follows: for some integer $s\geq 0$ and relation $x_{s+1}u_{s+1}=\sum_{j=1}^sx_ju_j$, where $u_j\in M$, choose indeterminates $X_1,\dots,X_s$ and an integer $N\geq 1$, let $F=u_{s+1}-\sum_{j=1}^s x_jX_j$ and let $$M'=M[X_1,\dots,X_s]_{\leq N}/F\cdot R[X_1,\dots,X_s]_{\leq N-1},$$ where $M[X_1,\dots,X_s]=M\otimes_RR[X_1,\dots,X_s]$ and thus $M[X_1,\dots,X_s]_{\leq N}$ refers to polynomials of degree at most $N$ (with coefficients in $M$). The definition of $M^{\prime}$ makes sense since $F$ has degree one in $X_j$. It is readily seen that in $M'$, the relation $x_{s+1}u_{s+1}=\sum_{j=1}^sx_ju_j$ is trivialized in the sense that $u_{s+1}$ is contained in $(x_1,\dots,x_s)M'$ by construction. We shall refer to the integer $N$ as the {\it degree bound} of the partial algebra modification. We can then recursively define a sequence of partial algebra modifications of an $R$-module $M$.

Now given a local map of local rings $(R,\m)\to (S,\n)$ we can define a {\it double sequence of partial algebra modifications} of an $R$-module $M$ with respect to $R\to S$, a system of parameters $x_1,\dots,x_d$ of $R$ and a system of parameters $y_1,\dots,y_{d'}$ of $S$ as follows: we first form a sequence of partial algebra modifications of $M$ over $R$ with respect to $x_1,\dots,x_d$, say $M=M_0, M_1,\dots, M_r$, and then we form a sequence of partial algebra modifications $N_0=S\otimes_R M_r, N_1,\dots, N_s$ of $N_0$ over $S$ with respect to $y_1,\dots,y_{d'}$. When $M$ is an $R$-algebra, we call this double sequence {\it bad} if the image of $1\in M$ in $N_s$ is in $\n N_s$.

The following was essentially taken from \cite[Theorem 4.2]{HochsterBigCohen-Macaulayalgebrasindimensionthree}, and is one of the main ingredients in our construction.

\begin{theorem}
\label{theorem--partial algebra modification}
Let $(R,\m)\to (S,\n)$ be a local homomorphism of local rings. Then there exists a commutative diagram:
\[\xymatrix{
R \ar[r]\ar[d] & S\ar[d]\\
B \ar[r] & C
}\]
such that $B$ is a balanced big Cohen-Macaulay algebra for $R$ and $C$ is a balanced big Cohen-Macaulay algebra for $S$ if and only if there is no bad double sequence of partial algebra modifications of $R$ over $R\to S$ with respect to $x_1,\dots,x_d$ of $R$ and $y_1,\dots,y_{d'}$ of $S$.
\end{theorem}


This theorem is actually a bit stronger than \cite[Theorem 4.2]{HochsterBigCohen-Macaulayalgebrasindimensionthree}.  Whereas Hochster allows the system of parameters to vary throughout the double sequence, we fix a system of parameters of $R$ and $S$. But the idea of the proof is the same: one first constructs $B'$ as a direct limit of finite sequences of modifications of $R$ and then constructs $C'$ as a direct limit of finite sequences of modifications of $S\otimes_RB$ over $S$. It is readily seen that $x_1,\dots,x_d$ and $y_1,\dots,y_{d'}$ are improper-regular sequences on $B'$ and $C'$ respectively. To guarantee that $\m B'\neq B'$ and $\n C'\neq C'$ one needs precisely that there is no bad double sequence of partial algebra modifications over $R\to S$.
Now $B^{\prime}$  and $C^{\prime}$ are not balanced, but that problem is easily remedied.
We invoke \cite[Corollary 8.5.3]{BrunsHerzogCohenMacaulayrings} to note that $B'\to C'$ induces $B=\widehat{B'}^\m\to C=\widehat{C'}^\n$, a map of balanced big Cohen-Macaulay algebras of $R\to S$ .


\subsection*{Acknowledgement} It is a pleasure to thank Mel Hochster for many enjoyable discussions on the vanishing conjecture for maps of Tor and other homological conjectures. We would like to thank Bhargav Bhatt for explaining the basic theory of perfectoid spaces to us and for pointing out Remark \ref{remark--Bhatt}. We would also like to thank Kiran Kedlaya, Karl Schwede, Kazuma Shimomoto and Chris Skalit for very helpful discussions. The second author is partially supported by NSF Grant DMS \#1600198, and NSF CAREER Grant DMS \#1252860/1501102.

\section{Weakly functorial construction of integral perfectoid algebras}

\noindent\textbf{Notations}: Throughout this section, $(A,\m,k)$ will always be a complete and unramified regular local ring of mixed characteristic with $k$ perfect, i.e., $A\cong W(k)[[x_1,\dots,x_{d-1}]]$ where $W(k)$ is the ring of Witt vectors with coefficients in $k$. Let $K^\circ$ be the $p$-adic completion of $W(k)[p^{1/p^\infty}]$ and $K=K^\circ[1/p]$. Let $A_{\infty,0}$ be the $p$-adic completion of $A[p^{1/p^\infty},x_1^{1/p^\infty},\dots,x_{d-1}^{1/p^\infty}]$, which is an integral perfectoid $K^\circ$-algebra.

For any nonzero element $g\in A$, we let $A_{\infty,0}\to A_\infty$ be Andr\'{e}'s construction of integral perfectoid $K^\circ$-algebras (for example see \cite[Theorem 1.4]{BhattDirectsummandandDerivedvariant}): $A_\infty$ is almost faithfully flat over $A_{\infty,0}$ modulo $p$ such that $g$ admits a compatible system of $p^k$-th roots in $A_\infty$. $A_\infty\langle\frac{p^n}{g}\rangle$ will denote the integral perfectoid $K^\circ$-algebra, which is the ring of bounded functions on the rational subset $\{x\in X| |p^n|\leq |g(x)|\}$ where $X=\Spa(A_\infty[1/p], A_\infty)$ is the perfectoid space associated to $A_\infty$. Since $g$ admits a compatible system of $p^k$-th roots in $A_\infty$, $A_\infty\langle\frac{p^n}{g}\rangle$ can be described almost explicitly as the $p$-adic completion of $A_\infty[(\frac{p^n}{g})^{\frac{1}{p^\infty}}]$ \cite[Lemma 6.4]{ScholzePerfectoidspaces}.

We begin by observing the following:

\begin{lemma}
\label{lemma--functorial property of Andre's construction}
Suppose $g\neq 0$ in $A/x_1A$. Then we have a natural map $A_\infty\to (A/x_1A)_\infty$ sending $g^{1/p^k}$ to $\overline{g}^{1/p^k}$.
\end{lemma}
\begin{proof}
We first note that there are natural maps $$A_{\infty, 0}\to (A/x_1A)_{\infty, 0}\to (A/x_1A)_\infty$$ where the first map is simply obtained by killing $x_1^{1/p^\infty}$. Thus we have a map $$A_{\infty,0}\langle T^{1/p^\infty}\rangle \to (A/x_1A)_\infty$$ of integral perfectoid $K^\circ$-algebras sending $T^{1/p^k}$ to $\overline{g}^{1/p^k}$. Since $A_\infty$ is the ring of functions on the Zariski closed subset of $Y=\Spa(A_{\infty,0}\langle T^{1/p^\infty}\rangle[1/p], A_{\infty,0}\langle T^{1/p^\infty}\rangle)$ defined by $T-g$, the map $A_{\infty,0}\langle T^{1/p^\infty}\rangle \to (A/x_1A)_\infty$ induces a map $A_\infty\to (A/x_1A)_\infty$ sending $g^{1/p^k}$ to $\overline{g}^{1/p^k}$.
\end{proof}

\begin{lemma}
\label{lemma--choosing coordinates}
Let $(R,\m,k)$ be a complete normal local domain with $k$ perfect, and let $Q\subseteq R$ be a height one prime ideal. Suppose both $R$ and $R/Q$ have mixed characteristic. Then we can find a complete and unramified regular local ring $A\cong W(k)[[x_1,\dots,x_{d-1}]]$ with $A\to R$ a module-finite extension such that:
\begin{enumerate}
\item $Q\cap A=(x_1)$;
\item $A_{(x_1)}\to R_Q$ is essentially \'{e}tale.
\end{enumerate}
\end{lemma}
\begin{proof}
Let $\{P_i\}$ be all the minimal primes of $(p)$; they all have height one. Since $R/Q$ has mixed characteristic, $p\notin Q$. Thus $Q$ is not contained in any of the $P_i$. By prime avoidance we can choose $x\in Q$ that is not in $(\cup_iP_i)\cup Q^{(2)}$.  Thus the image of $x$ in $R_Q$ generates $QR_Q$ since $R$ is normal, and $p, x$ is part of a system of parameters of $R$.

By Cohen's structure theorem, there exists a complete and unramified regular local ring $A\cong W(k)[[x_1,\dots,x_{d-1}]]$ and a module-finite extension $A\to R$ such that the image of $x_1$ in $R$ is $x$. It is clear that $Q\cap A=(x_1)$ because $Q\cap A$ is a height one prime of $A$ that contains $(x_1)$, so it must be $(x_1)$. To see $A_{(x_1)}\to R_Q$ is essentially \'{e}tale, note that the image of $x_1$, $x$, generates the maximal ideal $QR_Q$ of $R_Q$ and the extension of residue fields $A_{(x_1)}/(x_1)A_{(x_1)}\to R_Q/QR_Q$ is finite separable since both fields have characteristic $0$ ($p$ is inverted when we localize). Thus $A_{(x_1)}\to R_Q$ is unramified. But it is clearly flat because $R_Q$ is $x_1$-torsion free. Therefore $A_{(x_1)}\to R_Q$ is essentially \'{e}tale.
\end{proof}

The following is the main result of this section. It is crucial in proving the version of weakly functorial balanced big Cohen-Macaulay algebras that we need.

\begin{lemma}
\label{lemma--commutative diagram at the almost level}
Let $(R,\m,k)$ be a complete normal local domain with $k$ perfect, and let $Q\subseteq R$ be a height one prime ideal. Suppose both $R$ and $R/Q$ have mixed characteristic. We pick $A\cong W(k)[[x_1,\dots,x_{d-1}]]$ such that $A\to R$ is a module-finite extension satisfying the conclusion of Lemma \ref{lemma--choosing coordinates}.
Then there exists an element $g\in A$, whose image is nonzero in $A/x_1A$, such that $A_g\to R_g$ and $(A/x_1A)_g\to (R/Q)_g$ are both finite \'{e}tale. Furthermore, for every $n>0$, we have a commutative diagram:
\[\xymatrix{
R \ar[r] \ar[d] & R/Q \ar[d] \\
R_{\infty,n} \ar[r] & (R/Q)_{\infty,n}
}
\]
where $R_{\infty,n}$ (resp., $(R/Q)_{\infty,n}$) is an integral perfectoid $K^\circ$-algebra that is almost finite \'{e}tale over $A_\infty\langle\frac{p^n}{g}\rangle$ (resp., $(A/x_1A)_\infty\langle\frac{p^n}{g}\rangle$).
\end{lemma}
\begin{proof}
Let $g\in A$ be the discriminant of the map $A\to R$; i.e., it defines the locus of $\Spec A$ such that the map $A\to R$ is not essentially \'{e}tale when localizing. Since $A_{(x_1)}\to R_Q$ is essentially \'{e}tale, $g$ is nonzero in $A/x_1A$. Since $x_1$ generates $Q$ when localizing at $Q$ and we know that $A_g\to R_g$ and hence $(A/x_1A)_g\to (R/x_1R)_g$ are finite \'{e}tale, replacing $g$ by a multiple we have $A_g\to R_g$ and $(A/x_1A)_g\to (R/Q)_g$ are both finite \'{e}tale.
By Lemma \ref{lemma--choosing coordinates} we have a commutative diagram:
\[\xymatrix{
A \ar[r]\ar[d] & {R} \ar[d] \\
{A/x_1A} \ar[r] & R/Q.
}
\]
By Lemma \ref{lemma--functorial property of Andre's construction} we also have a commutative diagram:
\[\xymatrix{
A \ar[r]\ar[d] & A_\infty  \ar[r]\ar[d] & A_\infty\langle\frac{p^n}{g}\rangle \ar[d]\\
A/x_1A \ar[r]  & (A/x_1A)_\infty \ar[r] &  (A/x_1A)_\infty\langle\frac{p^n}{g}\rangle.
}
\]
Tensoring over $A$ we get a natural commutative diagram:
\[\xymatrix{
R \ar[r]\ar[d] & R\otimes A_\infty\langle\frac{p^n}{g}\rangle  \ar[d] \\
R/Q \ar[r] & (R/Q) \otimes (A/x_1A)_\infty\langle\frac{p^n}{g}\rangle.
}
\]
Since $A_g\to R_g$ and $(A/x_1A)_g\to (R/Q)_g$ are both finite \'{e}tale and $g$ divides $p^n$ in $A_\infty\langle\frac{p^n}{g}\rangle$ and $(A/x_1A)_\infty\langle\frac{p^n}{g}\rangle$, we know that $(R\otimes A_\infty\langle\frac{p^n}{g}\rangle)[1/p]$ and $((R/Q) \otimes (A/x_1A)_\infty\langle\frac{p^n}{g}\rangle)[1/p]$ are finite \'{e}tale over $(A_\infty\langle\frac{p^n}{g}\rangle)[1/p]$ and $((A/x_1A)_\infty\langle\frac{p^n}{g}\rangle)[1/p]$ respectively. Therefore $$(R\otimes A_\infty\langle\frac{p^n}{g}\rangle)[1/p]\to ((R/Q) \otimes (A/x_1A)_\infty\langle\frac{p^n}{g}\rangle)[1/p]$$ is a morphism  of perfectoid $K$-algebras; thus it induces a map on the ring of powerbounded elements $$R_{\infty,n}:=(R\otimes A_\infty\langle\frac{p^n}{g}\rangle)[1/p]^\circ\to (R/Q)_{\infty,n}:=((R/Q) \otimes (A/x_1A)_\infty\langle\frac{p^n}{g}\rangle)[1/p]^\circ.$$
The almost purity theorem \cite[Theorem 7.9]{ScholzePerfectoidspaces} implies that $R_{\infty,n}$ and $(R/Q)_{\infty,n}$ are integral perfectoid $K^\circ$-algebras that are almost finite \'{e}tale over $A_\infty\langle\frac{p^n}{g}\rangle$ and $(A/x_1A)_\infty\langle\frac{p^n}{g}\rangle$ respectively. Therefore we have a commutative diagram
\[\xymatrix{
R \ar[r] \ar[d] & R/Q \ar[d] \\
R_{\infty,n} \ar[r] & (R/Q)_{\infty,n}
}
\]
as desired.
\end{proof}

\section{The main result}

In this section we continue to use the notation from the beginning of Section 2. The main theorem we want to prove is the following:

\begin{theorem}
\label{theorem--big CM}
Let $(R,\m,k)$ be a complete local domain with $k$ algebraically closed, and let $Q\subseteq R$ be a height one prime ideal. Suppose both $R$ and $R/Q$ have mixed characteristic. Then there exists a commutative diagram:
\[\xymatrix{
R \ar[r] \ar[d] & R/Q \ar[d] \\
B \ar[r] & C
}
\]
where $B$, $C$ are balanced big Cohen-Macaulay algebras for $R$ and $R/Q$ respectively.
\end{theorem}

To prove this we need several lemmas.

\begin{lemma}
\label{lemma-forcing elements in Ainfinity}
Let $A\cong W(k)[[x_1,\dots,x_{d-1}]]$ be a complete and unramified regular local ring with $k$ perfect, and let $I=(p, y_1, \dots,y_s)$ be an ideal of $A$ that contains $p$. Fix a nonzero element $g=p^mg_0\in A$ where $p\nmid g_0$, and consider the extension $A\to A_\infty\to A_\infty\langle \frac{p^n}{g}\rangle$. Suppose $z\in IA_\infty\langle \frac{p^n}{g}\rangle\cap A_\infty$ for some $n>p^a+m$ (one should think that $n\gg p^a\gg0$ here), then we have $(pg)^{1/p^a}z\in IA_\infty$.
\end{lemma}
\begin{proof}
Using the almost explicit description of $A_\infty\langle \frac{p^n}{g}\rangle$ \cite[Lemma 6.4]{ScholzePerfectoidspaces}, we have $$p^{1/p^t}z\in I\widehat{A_\infty[(\frac{p^n}{g})^{1/p^{\infty}}]}$$ for some $t>a$. This implies that the image of $p^{1/p^t}z$ in $A_\infty[(\frac{p^n}{g})^{1/p^{\infty}}]/p=\widehat{A_\infty[(\frac{p^n}{g})^{1/p^{\infty}}]}/p$ is contained in the ideal $(y_1,\dots, y_s)$. Therefore we can write
$$p^{1/p^t}z=pf_0+y_1f_1+\cdots+y_sf_s$$
where $f_0, f_1,\dots,f_s\in A_\infty[(\frac{p^n}{g})^{1/p^{\infty}}]$. Then there exists integers $k$ and $h$ such that $f_0, f_1,\dots,f_s$ are elements in $A_\infty[(\frac{p^n}{g})^{1/p^k}]$ of degree bounded by $p^kh$. Multiplying by $g_0^h$ to clear all the denominators in $f_i$, one gets:
$$p^{1/p^t}g_0^hz\in(g_0^{h-(1/p^{a})},p^{(n-m)/p^a})\cdot (p, y_1,\dots,y_s)A_\infty.$$
From this we know:
$$p^{1/p^t}g_0^hz=g_0^{h-(1/p^a)}(ph_0+y_1h_1+\cdots+y_sh_s) \text{ in } A_\infty/p^{(n-m)/p^{a}},$$
where $h_0,h_1,\dots,h_s\in A_\infty$. Rewriting this we have
$$g_0^{h-(1/p^a)}(p^{1/p^t}g_0^{1/p^a}z-ph_0-y_1h_1-\cdots-y_sh_s)=0 \text{ in } A_\infty/p^{(n-m)/p^{a}}.$$
Since $p\nmid g_0$, $g_0$ is a nonzerodivisor on $A/p$. This implies $g_0^{h-(1/p^a)}$ is an almost nonzerodivisor on $A_\infty/p^{(n-m)/p^{a}}$ since $A\to A_{\infty, 0}$ is faithfully flat and $A_{\infty, 0}\to A_\infty$ is almost faithfully flat modulo $p$. Hence $p^{1/p^t}g_0^{1/p^a}z-ph_0-y_1h_1-\cdots-y_sh_s$ is killed by $(p^{1/p^{\infty}})$. In particular, since $t>a$, we know $$(pg_0)^{1/p^{a}}z\in (p, y_1,\dots, y_s) \text{ in } A_\infty/p^{(n-m)/p^{a}}.$$ Finally, since $n>p^a+m$ and $g$ is a multiple of $g_0$, we have $$(pg)^{1/p^a}z\in(p,y_1,\dots,y_s)A_\infty.$$
This finishes the proof.
\end{proof}

\begin{lemma}
\label{lemma-forcing elements in pA_infty}
Let $A\cong W(k)[[x_1,\dots,x_{d-1}]]$ be a complete and unramified regular local ring with $k$ perfect. Fix a nonzero element $g=p^mg_0\in A$ where $p\nmid g_0$ and $n>m$. Suppose $z\in A_\infty[(\frac{p^n}{g})^{1/p^\infty}]$ and $p^Dz\in A_\infty$ for some $D>0$. Then $p^Dz\in p^{D-(1/p^t)}A_\infty$  for all $t$.
\end{lemma}
\begin{proof}
There exist $k\gg0$ such that $z\in A_\infty[(\frac{p^n}{g})^{1/p^k}]$. Choosing a high enough power of $g_0$ to clear denominators, we get $g_0^hz\in A_\infty$. So $g_0^h(p^Dz)\in p^DA_\infty$. Since $g_0$ is a nonzerodivisor on $A/p^D$ and $A_\infty/p^D$ is almost faithfully flat over $A/p^D$, $p^{1/p^t}p^Dz\in p^DA_\infty$ for all $t$. Since $A_\infty$ is $p$-torsion free, $p^Dz\in p^{D-(1/p^t)}A_\infty$  for all $t$.
\end{proof}

\begin{lemma}
\label{lemma-annihilator of almost projective module}
Let $A\cong W(k)[[x_1,\dots,x_{d-1}]]$ be a complete and unramified regular local ring with $k$ perfect. Fix a nonzero element $g=p^mg_0\in A$ where $p\nmid g_0$, and consider the extension $A\to A_\infty\to A_\infty\langle \frac{p^n}{g}\rangle$ for every $n$. Suppose $S$ is an almost finite projective $A_\infty\langle \frac{p^n}{g}\rangle$-algebra. If $p^a+m<n$, then we have $(p^{1/p^{\infty}})(pg)^{1/p^a}$ annihilates $\frac{(p,x_1,\dots,x_s)S : x_{s+1}}{(p,x_1,\dots,x_s)S}$ for all $s< d-1$.
\end{lemma}

\begin{proof}
Suppose $y\in(p,x_1,\dots,x_s){A_\infty\langle \frac{p^n}{g}\rangle}:x_{s+1}$. Since $y$ is an element of $A_\infty\langle \frac{p^n}{g}\rangle$, for every $t>0$, $p^{1/p^t}y\in \widehat{A_\infty[(\frac{p^n}{g})^{1/p^\infty}]}$ and $x_{s+1}p^{1/p^t}y\in(p,x_1,\dots,x_s)\widehat{A_\infty[(\frac{p^n}{g})^{1/p^\infty}]}$ by \cite[Lemma 6.4]{ScholzePerfectoidspaces}. Thus modulo $p$, $p^{1/p^t}y$ gives an element in $\widehat{A_\infty[(\frac{p^n}{g})^{1/p^\infty}]}/p=A_\infty[(\frac{p^n}{g})^{1/p^\infty}]/p$. We pick $z\in A_\infty[(\frac{p^n}{g})^{1/p^\infty}]$ such that $z\equiv p^{1/p^t}y$ modulo $p\widehat{A_\infty[(\frac{p^n}{g})^{1/p^\infty}]}$.

Now $x_{s+1}z$ is an element of $A_\infty[(\frac{p^n}{g})^{1/p^\infty}]$ whose image in $A_\infty[(\frac{p^n}{g})^{1/p^\infty}]/p=\widehat{A_\infty[(\frac{p^n}{g})^{1/p^\infty}]}/p$ is contained in the ideal $(x_1,\dots,x_s)(\widehat{A_\infty[(\frac{p^n}{g})^{1/p^\infty}]}/p)$. Therefore we know $$x_{s+1}z\in (p,x_1,\dots,x_s){A_\infty[(\frac{p^n}{g})^{1/p^\infty}]}$$ 
and thus $$z\in (p,x_1,\dots,x_s)A_\infty[(\frac{p^n}{g})^{1/p^\infty}] :x_{s+1}.$$

Next we write $z=u+(\frac{p^n}{g})^{1/p^a}u'$ such that $g_0^{1/p^a}u\in A_\infty$, $u'\in A_\infty[(\frac{p^n}{g})^{1/p^\infty}]$, and we also write $x_{s+1}z=v+(\frac{p^n}{g})^{1/p^a}v'$ such that $g_0^{1/p^a}v\in (p,x_1,\dots,x_s)A_\infty$, $v'\in (p,x_1,\dots,x_s)A_\infty[(\frac{p^n}{g})^{1/p^\infty}]$. We consider two expressions of $x_{s+1}g_0^{1/p^a}z$: $$x_{s+1}g_0^{1/p^a}u+p^{(n-m)/p^a}x_{s+1}u'=x_{s+1}g_0^{1/p^a}z=g_0^{1/p^a}v+p^{(n-m)/p^a}v'.$$ From this we know that
\begin{equation}
\label{equation--expression involving uvu'v'}
x_{s+1}(g_0^{1/p^a}u)=g_0^{1/p^a}v+p^{(n-m)/p^a}(v'-x_{s+1}u')
\end{equation}
It follows from (\ref{equation--expression involving uvu'v'}) that $p^{(n-m)/p^a}(v'-x_{s+1}u')\in A_\infty$ (since the other two terms are in $A_\infty$). Thus by Lemma \ref{lemma-forcing elements in pA_infty}, $p^{(n-m)/p^a}(v'-x_{s+1}u')\in pA_\infty$ since $n>p^a+m$. But now (\ref{equation--expression involving uvu'v'}) tells us that
$$x_{s+1}(g_0^{1/p^a}u) \in (p,x_1,\dots,x_s)A_\infty+pA_\infty = (p,x_1,\dots,x_s)A_\infty.$$
Since $g_0^{1/p^a}u\in A_\infty$ and $p, x_1,\dots,x_{s+1}$ is an almost regular sequence on $A_\infty$, $$(pg_0)^{1/p^a}u\in (p,x_1,\dots,x_s)A_\infty.$$ But now $$(pg_0)^{1/p^a}z=(pg_0)^{1/p^a}u+p^{1/p^a}p^{(n-m)/p^a}u'.$$
Therefore
$$(pg_0)^{1/p^a}z\in (p,x_1,\dots,x_s)A_\infty+pA_\infty[(\frac{p^n}{g})^{1/p^\infty}] \subseteq (p,x_1,\dots,x_s)A_\infty[(\frac{p^n}{g})^{1/p^\infty}].$$
Because $z\equiv p^{1/p^t}y$ modulo $p\widehat{A_\infty[(\frac{p^n}{g})^{1/p^\infty}]}$ and $g$ is a multiple of $g_0$, we have
$$p^{1/p^t}(pg)^{1/p^a}y\in (p,x_1,\dots,x_s)\widehat{A_\infty[(\frac{p^n}{g})^{1/p^\infty}]}\subseteq (p,x_1,\dots,x_s)A_\infty\langle \frac{p^n}{g}\rangle.$$
Since this is true for all $t>0$, we have $(p^{1/p^{\infty}})(pg)^{1/p^a}$ annihilates $\frac{(p,x_1,\dots,x_s){A_\infty\langle \frac{p^n}{g}\rangle}:x_{s+1}}{(p,x_1,\dots,x_s){A_\infty\langle \frac{p^n}{g}\rangle}}$.
 Finally, since $S$ is an almost finite projective $A_\infty\langle \frac{p^n}{g}\rangle$-algebra, by \cite[Lemma 2.4.31]{GabberRameroAlmostringtheory}, $$\frac{(p,x_1,\dots,x_s)S : x_{s+1}}{(p,x_1,\dots,x_s)S}=\Hom_S(S/x_{s+1}, S/(p,x_1,\dots,x_s))$$ is almost isomorphic to $$S\otimes\Hom_{A_\infty\langle \frac{p^n}{g}\rangle}\left(A_\infty\langle \frac{p^n}{g}\rangle/x_{s+1}, A_\infty\langle \frac{p^n}{g}\rangle/(p,x_1,\dots,x_s)\right)=S\otimes \frac{(p,x_1,\dots,x_s){A_\infty\langle \frac{p^n}{g}\rangle}:x_{s+1}}{(p,x_1,\dots,x_s){A_\infty\langle \frac{p^n}{g}\rangle}}.$$
Therefore $(p^{1/p^{\infty}})(pg)^{1/p^a}$ annihilates $\frac{(p,x_1,\dots,x_s)S : x_{s+1}}{(p,x_1,\dots,x_s)S}$ as well.
\end{proof}


We need the following lemma from \cite{HochsterBigCohen-Macaulayalgebrasindimensionthree}:

\begin{lemma}[Lemma 5.1 of \cite{HochsterBigCohen-Macaulayalgebrasindimensionthree}]
\label{lemma-Hochster}
Let $M$ be an $R$-module and let $T$ be an $R$-algebra with a map $\alpha$: $M\to T[1/c]$. Let $M\to M'$ be a partial algebra modification of $M$ with respect to part of a system of parameters $p, x_1,\dots, x_s, x_{s+1}$ with degree bound $D$. Suppose $x_{s+1}t_{s+1}=pt_0+x_1t_1+\cdots+x_st_s$ with $t_j\in T$ implies $ct_{s+1}\in (p, x_1,\dots,x_s)T$ and $\alpha(M)\subseteq c^{-N}T$. Then there is an $R$-linear map $\beta$: $M'\to T[1/c]$ extending $\alpha$ with image contained in $c^{-N'}T$ where $N'=ND+N+D$ depends only on $N$ and $D$.
\end{lemma}

\begin{proof}[Proof of Theorem \ref{theorem--big CM}]
Let $R'$ be the normalization of $R$ and let $Q'$ be a height one prime of $R'$ that lies over $Q$. Note that the residue field of $R'$ is still $k$ since we assumed $k$ is algebraically closed. If we can construct weakly functorial big Cohen-Macaulay algebras for $R'\to R'/Q'$ then the same follows for $R\to R/Q$. Thus without loss of generality we can assume $(R,\m,k)$ is normal. Let
\[\xymatrix{
R \ar[r] \ar[d] & R/Q \ar[d] \\
R_{\infty,n} \ar[r] & (R/Q)_{\infty,n}
}
\]
be the commutative diagram constructed in Lemma \ref{lemma--commutative diagram at the almost level}. Moreover, abusing notation slightly, suppose $g=p^{m_1}g_0$ in $A$ and $\overline{g}=p^{m_2}\overline{g}_0$ in $A/x_1A$ such that $p\nmid g_0$ and $p\nmid \overline{g}_0$.  

 Now $R_{\infty,n}$ and $(R/Q)_{\infty,n}$ are almost finite \'{e}tale over $A_\infty\langle\frac{p^n}{g}\rangle$ and $(A/x_1A)_\infty\langle\frac{p^n}{g}\rangle$ respectively, in particular they are almost finite projective over $A_\infty\langle\frac{p^n}{g}\rangle$ and $(A/x_1A)_\infty\langle\frac{p^n}{g}\rangle$ respectively (see \cite[Definition 4.3 and Proposition 4.10]{ScholzePerfectoidspaces}). Lemma \ref{lemma-annihilator of almost projective module} shows that, for every $n$ and $p^a$ such that $n> p^a+m_1+m_2$, with $c=(pg)^{2/p^a}$, if $x_{s+1}t_{s+1}=pt_0+x_1t_1+\cdots+x_st_s$ with $t_j\in R_{\infty,n}$ (resp., $(R/Q)_{\infty,n}$), we have that $ct_{s+1}\in (p, x_1,\dots,x_s)R_{\infty,n}$ (resp., $(R/Q)_{\infty,n}$).

By Theorem \ref{theorem--partial algebra modification}, it suffices to show that there is no bad double sequence of partial algebra modifications of $R$. Suppose there is one:
$$R\to M_1\to\cdots \to M_r\to(R/Q)\otimes M_r\to N_1\to\cdots\to N_s.$$
We claim that there exists a commutative diagram:
\tiny\tiny
\[\xymatrix{
R \ar[r] \ar[d] & M_1 \ar[r]\ar[d] &  \cdots \ar[r] & M_r \ar[d]^\alpha \ar[r] & (R/Q)\otimes M_r \ar[r] \ar[d] & N_1\ar[r] \ar[d] &\cdots \ar[r] & N_s\ar[d]^\beta \\
R_{\infty,n}[1/c] \ar[r]^-= & R_{\infty,n}[1/c] \ar[r]^-= & \cdots \ar[r]^-= & R_{\infty,n}[1/c]\ar[r] & (R/Q)_{\infty,n}[1/c] \ar[r]^-= & (R/Q)_{\infty,n}[1/c]\ar[r]^-=  &\cdots \ar[r]^-= & (R/Q)_{\infty,n}[1/c]
}
\]

\normalsize

The leftmost vertical map is the natural one; the first half of the diagram exists by Lemma \ref{lemma-Hochster}; the middle commutative diagram exists because the composite map $M_r\to R_{\infty,n}[1/c]\to (R/Q)_{\infty,n}[1/c]$ induces a map $(R/Q)\otimes M_r\to (R/Q)_{\infty,n}[1/c]$ since $(R/Q)_{\infty,n}[1/c]$ is an $R/Q$-algebra; the second half of the diagram exists by Lemma \ref{lemma-Hochster} again.

Let $D>0$ be an integer larger than the degree bounds for all the partial algebra modifications in this sequence. Applying Lemma \ref{lemma-Hochster} repeatedly to the first half of the diagram, we know there is an integer $M$ depending only on $D$, but not on $n$ and $p^a$, such that the image of $\alpha$ is contained in $c^{-M}R_{\infty,n}$. The image of the map $(R/Q)\otimes M_r\to (R/Q)_{\infty,n}[1/c]$ is contained in $c^{-M}(R/Q)_{\infty,n}$ because $R_{\infty,n}[1/c]\to (R/Q)_{\infty,n}[1/c]$ is induced by $R_{\infty,n}\to (R/Q)_{\infty,n}$. But then applying Lemma \ref{lemma-Hochster} repeatedly to the second half of the diagram, we know that there exists an integer $N$ depending on $M$ and $D$ (and hence only on $D$), but not on $n$ and $p^a$, such that the image of $\beta$ is contained in $c^{-N}(R/Q)_{\infty,n}$.

Now we chase the above diagram and we see that on the one hand, the element $1\in R$ maps to $1\in (R/Q)_{\infty,n}[1/c]$. But on the other hand, since the sequence is bad, the image of $1\in R$ in $N_s$ is in $\m N_s$ and hence the image of $1\in R$ is contained in $\m c^{-N}(R/Q)_{\infty,n}$ in $(R/Q)_{\infty,n}[1/c]$. Therefore we have $1\in\m((pg)^{2/p^a})^{-N}(R/Q)_{\infty,n}$, that is, $$(pg)^{2N/p^a}\in\m (R/Q)_{\infty,n}.$$ Because $\m$ is the maximal ideal of $R$ and $A=W(k)[[x_1,\dots,x_{d-1}]]\to R$ is module-finite, $\m^{N'}\subseteq (p, x_1,\dots,x_{d-1})R$ for some fixed $N'$. We thus have:
$$(pg)^{2NN'/p^a}\in(p,x_2,\dots,x_{d-1})(R/Q)_{\infty,n}.$$
Since $(R/Q)_{\infty,n}$ is almost finite \'{e}tale over $(A/x_1A)_\infty\langle \frac{p^n}{g}\rangle$, we know that
$$(pg)^{(2NN'+1)/p^a}\in (p,x_2,\dots,x_{d-1}) (A/x_1A)_\infty\langle \frac{p^n}{g}\rangle \cap (A/x_1A)_\infty.$$
But now Lemma \ref{lemma-forcing elements in Ainfinity} implies $(pg)^{(2NN'+2)/p^a}\in(p,x_2,\dots,x_{d-1})(A/x_1A)_\infty$ for all $p^a$. Because $N, N'$ do not depend on $p^a$, we know that $pg\in (p,x_2,\dots,x_{d-1})^{m}(A/x_1A)_\infty$ for all $m>0$. Since $(A/x_1A)_\infty$ is almost faithfully flat over $(A/x_1A)_{\infty,0}$ mod $p^m$, we know that $$p^2g\in (p,x_2,\dots,x_{d-1})^{m}(A/x_1A)_{\infty,0}\cap (A/x_1A)=(p,x_2,\dots,x_{d-1})^m(A/x_1A)$$ for all $m>0$ by faithfull flatness of $(A/x_1A)_{\infty, 0}$ over $A/x_1A$. But then $$p^2g\in \cap_m(p,x_2,\dots,x_{d-1})^m(A/x_1A)=0,$$ which is a contradiction.
\end{proof}

\begin{remark}
\label{remark--Bhatt}
We point out that the quantitative form of Scholze's Hebbarkeitssatz \cite[Theorem 4.2]{BhattDirectsummandandDerivedvariant} implies Lemma \ref{lemma-forcing elements in Ainfinity} and the following weaker form of Lemma \ref{lemma-annihilator of almost projective module}: if $\{S_n\}_n$ is a pro-system such that $S_n$ is an almost finite projective $A_\infty\langle\frac{p^n}{g}\rangle$-algebra, then for every $k\geq 1$ and $n\geq p^a+m$, $(p^{1/p^\infty})(pg)^{1/p^a}$ annihilates the image of $\frac{(p,x_1,\dots,x_s)S_{k+n} : x_{s+1}}{(p,x_1,\dots,x_s)S_{k+n}}$ in $\frac{(p,x_1,\dots,x_s)S_{k} : x_{s+1}}{(p,x_1,\dots,x_s)S_{k}}$. This weaker form is enough to establish Theorem \ref{theorem--big CM}, but one needs to modify the proof of Lemma \ref{lemma-Hochster} and Theorem \ref{theorem--big CM}: to extend each partial algebra modification to $R_{\infty,n}[1/c]$ one needs to decrease $n$ roughly by $p^a$ in order to trivialize bad relations (and keep control on the denominators). We leave it to the interested reader to carry out the details.
\end{remark}

\section{Applications}
The results obtained in the preceding section are strong enough to establish the mixed-characteristic case of Hochster-Huneke's vanishing conjecture for maps of Tor \cite{HochsterHunekeApplicationsofbigCMalgebras}.

\begin{theorem}
\label{theorem--vanishing theorem for maps of Tor}
Let $A\to R\to S$ be maps of Noetherian rings such that $A\to S$ is a local homomorphism of mixed characteristic regular local rings and $R$ is a module-finite torsion-free extension of $A$. Then for all $A$-modules $M$, the map $\Tor_i^A(M, R)\to \Tor_i^A(M, S)$ vanishes for all $i\geq 1$.
\end{theorem}

We need the following important reduction. This reduction is known to experts and is proved implicitly in \cite[Chapter 5.2]{Ranganathanthesis} and \cite[Section 13]{HochsterHomologicalConjecutersLimCMsequences}. We will give a sketch of the proof.

\begin{lemma}
\label{lemma--Ranganathan}
To prove Theorem \ref{theorem--vanishing theorem for maps of Tor}, we can assume $(A,\m)$ is complete, $R$ is a complete local domain, and $S=A/xA$ where $x\in \m-\m^2$.
\end{lemma}
\begin{proof}[Sketch of proof]
We can assume $M$ is finitely generated. Replacing $M$ by its first module of syzygies over $A$ repeatedly, we only need to prove the case $i=1$. We may further assume $M=A/I$ by \cite[Lemma 5.2.1]{Ranganathanthesis} or \cite[Page 15]{HochsterHomologicalConjecutersLimCMsequences}.\footnote{In this process we may lose $A$ and $S$ being local, but we can always localize $A$ and $S$ again to assume they are local (and have mixed characteristic, since otherwise Theorem \ref{theorem--vanishing theorem for maps of Tor} is known).} Next by \cite[(4.5)(a)]{HochsterHunekeApplicationsofbigCMalgebras}, we can assume $A$ and $S$ are both complete, $R$ is a complete local domain, and $A\to S$ is surjective; i.e., $S=A/P$ where $P$ is generated by part of a regular system of parameters of $A$ (note that $p\notin P$ since $S$ has mixed characteristic). It follows that $S=R/Q$ for some prime ideal $Q$ of $R$ lying over $P$. After all these reductions, we note that by \cite[Lemma 13.6]{HochsterHomologicalConjecutersLimCMsequences}, $\Tor_1^A(A/I, R)\to \Tor_1^A(A/I, S)$ vanishes if and only if $IQ\cap P=IP$.

We next want to reduce to the case that $P$ is generated by one element. The trick is to replace $A$ by its extended Rees ring $\widetilde{A}=A[Pt, t^{-1}]$, $R$ by $\widetilde{R}=R[Pt, t^{-1}]$ and $S$ by $\widetilde{S}=\widetilde{A}/t^{-1}\widetilde{A}$. Since $P$ is generated by part of a regular system of parameters, $\widetilde{A}$ and $\widetilde{S}$ are still regular. The point is that there is a homogeneous prime ideal $\widetilde{Q}\subseteq\widetilde{R}$ that contains $Q$ and contracts to $t^{-1}\widetilde{A}\subseteq \widetilde{A}$ (see \cite[Proof of Theorem 5.2.6]{Ranganathanthesis} or \cite[Page 16]{HochsterHomologicalConjecutersLimCMsequences}), thus we have $\widetilde{A}\to \widetilde{R}\to\widetilde{S}$. Therefore if we can prove Theorem \ref{theorem--vanishing theorem for maps of Tor} for $\widetilde{A}\to \widetilde{R}\to\widetilde{S}$ and $M=\widetilde{A}/I\widetilde{A}$, then \cite[Lemma 13.6]{HochsterHomologicalConjecutersLimCMsequences} implies that $I\widetilde{Q}\cap t^{-1}\widetilde{A}=It^{-1}\widetilde{A}$. Comparing the degree $0$ part, we see that $IQ\cap P=IP$.

Finally, we can localize $\widetilde{A}$ and $\widetilde{S}$ and complete, and reduce to the case $\widetilde{R}$ is a complete local domain as in \cite[(4.5)(a)]{HochsterHunekeApplicationsofbigCMalgebras}. Note that $\widetilde{S}$ is obtained from $\widetilde{A}$ by killing one element (and we may assume $\widetilde{S}$ still has mixed characteristic after localization). We thus obtain all the desired reductions.
\end{proof}


\begin{proof}[Proof of Theorem \ref{theorem--vanishing theorem for maps of Tor}]
By Lemma \ref{lemma--Ranganathan}, we may assume $R$ is a complete local domain and $S=A/xA$. It follows that $S=R/Q$ for a height one prime $Q$ of $R$. Since $A\to S$ and $R\to S$ are both surjective, $A$, $R$, $S$ have the same residue field $k$. We fix a coefficient ring $W(k)$ of $A$, then the images of $W(k)$ in $R$ and $S$ are also coefficient rings of $R$ and $S$. Replacing $A$, $R$, $S$ by their faithfully flat extensions $A\widehat{\otimes}_{W(k)}W(\overline{k})$, $R\widehat{\otimes}_{W(k)}W(\overline{k})$, $S\widehat{\otimes}_{W(k)}W(\overline{k})$ does not affect whether the map on Tor vanishes or not. Thus without loss of generality we may assume $k$ is algebraically closed.

By Theorem \ref{theorem--big CM}, we have a commutative diagram:
\[\xymatrix{
R \ar[r] \ar[d]& S=R/Q \ar[d]\\
B\ar[r]  & C
}
\]
where $B$ and $C$ are balanced big Cohen-Macaulay algebras for $R$ and $S$ respectively. This induces a commutative diagram:
\[\xymatrix{
\Tor_i^A(M, R) \ar[r] \ar[d]& \Tor_i^A(M, S)\ar[d]\\
\Tor_i^A(M, B) \ar[r] & \Tor_i^A(M, C)
}
\]
Since $B$ is a balanced big Cohen-Macaulay algebra over $R$ (and hence also over $A$), it is faithfully flat over $A$ so $\Tor_i^A(M, B)=0$ for all $i\geq 1$. Moreover, $C$ is faithfully flat over $S$ since it is a balanced big Cohen-Macaulay {algebra} over $S$ and $S$ is regular, thus $\Tor_i^A(M, S)\to \Tor_i^A(M, C)$ is injective. Chasing the diagram above we know that the map $\Tor_i^A(M, R)\to \Tor_i^A(M, S)$ vanishes for all $i\geq 1$.
\end{proof}

A local ring $(R,\m)$ of dimension $d$ is called {\it pseudo-rational} if it is normal, Cohen-Macaulay, analytically unramified (i.e., the completion $\widehat{R}$ is reduced), and if for every projective and birational map  $\pi$: $W\to \Spec R$, the canonical map $H_\m^d(R)\to H_E^d(W, O_W)$ is injective where $E=\pi^{-1}(\m)$ denotes the closed fibre. In characteristic $0$, pseudo-rational singularities are the same as rational singularities. Very recently, Kov\'{a}cs \cite{KovacsSuperRational} has proved a remarkable result that, in {\it all} characteristics, if $\pi$: $X\to\Spec R$ is projective and birational, where $X$ is Cohen-Macaulay and $R$ is pseudo-rational, then $\mathbf{R}\pi_*O_X=R$.

In equal characteristic, direct summands of regular rings are pseudo-rational \cite{BoutotRationalsingularitiesandquotientsbyreductivegroups} \cite{HochsterHunekeTC1}. This is usually called Boutot's theorem. It is well known that the vanishing conjecture for maps of Tor in a given characteristic implies that direct summands of regular rings are Cohen-Macaulay \cite[(4.3)]{HochsterHunekeApplicationsofbigCMalgebras}. What we want to prove next is the analog of Boutot's theorem that direct summands of regular rings are pseudo-rational in mixed characteristic. This in fact also follows formally from the vanishing conjecture for maps of Tor \cite{MaThevanishingconjectureformapsofTorandderivedsplinters}. Since the full details were not written down explicitly in \cite{MaThevanishingconjectureformapsofTorandderivedsplinters}, we give a complete argument here. We first recall the following Sancho de Salas exact sequence \cite{SanchodeSalasBlowingupmorphismswithCohenMacaulayassociatedgradedrings}.

Let $T=R[Jt]=R\oplus Jt\oplus J^2t^2\oplus\cdots$ and let $W=\Proj T\to \Spec R$ be the blow up with $E=\pi^{-1}(\m)$. Pick $f_1,\dots,f_n\in Jt=[T]_1$ such that $U=\{U_i=\Spec[T_{f_i}]_0\}$ is an affine open cover of $W$. We have an exact sequence of chain complexes:
\[0\to \check{C}^\bullet(U, O_W)[-1]\to [C^\bullet(f_1,\dots,f_n, T)]_0\to R\to 0.\]
Since $\check{C}^\bullet(U, O_W)\cong\mathbf{R}\pi_*O_W$ and $C^\bullet(f_1,\dots,f_n, T)=[\mathbf{R}\Gamma_{T_{>0}}T]_0$ the above sequence gives us (after rotating) an exact triangle:
\[[\mathbf{R}\Gamma_{T_{>0}}T]_0\to R\to \mathbf{R}\pi_*O_W\xrightarrow{+1}\]
Applying $\mathbf{R}\Gamma_{\m}$, we get:
\[
[\mathbf{R}\Gamma_{\m+T_{>0}}T]_0 \to  \mathbf{R}\Gamma_{\m}R  \to \mathbf{R}\Gamma_{\m}\mathbf{R}\pi_*O_W \xrightarrow{+1}.
\]
Taking cohomology we get the Sancho de Salas exact sequence:
\begin{equation}
\label{equation--sancho de salas}
\xymatrix{
{ } \ar[r] & h^d([\mathbf{R}\Gamma_{\m+T_{>0}}T]_0) \ar[r] \ar[d]^= &  h^d(\mathbf{R}\Gamma_{\m}R) \ar[r] \ar[d]^= & h^d(\mathbf{R}\Gamma_{\m}\mathbf{R}\pi_*O_W)\ar[d]^= \ar[r] & { } \\
{ } \ar[r] & [H_{\m+T_{>0}}^d(T)]_0 \ar[r] & H^d_\m(R) \ar[r] & H^d_E(W, O_W) \ar[r] & {}.
}
\end{equation}

We also recall that $R\to S$ is {\it pure} if $R\otimes M\to S\otimes M$ is injective for every $R$-module $M$. This is slightly weaker than saying that $R\to S$ splits as a map of $R$-modules. If $R$ is an $A$-algebra and $R\to S$ is pure, then $\Tor_i^A(M, R)\to \Tor_i^A(M, S)$ is injective for every $i$ \cite[(2.1)(h)]{HochsterHunekeApplicationsofbigCMalgebras}, in particular, $H_\m^i(R)\to H_\m^i(S)$ is injective for every $i$.

We are ready to prove the following corollary. We state the result in the local setting, but the general case reduces immediately to the local case.

\begin{corollary}
\label{corollary--mixed char Boutot}
Let $(R,\m)\to (S,\n)$ be a pure map of local rings such that $(S,\n)$ is regular of mixed characteristic. Then $R$ is pseudo-rational. In particular, direct summands of regular rings are pseudo-rational.
\end{corollary}
\begin{proof}
We can complete $R$ and $S$ at $\m$ and $\n$ respectively to assume both $R$ and $S$ are complete. $R$ is normal since pure subrings of normal domains are normal. By Cohen's structure theorem, we have a module-finite extension $A\to R$ such that $A$ is a complete regular local ring. Let $x_1,\dots,x_d$ be a regular system of parameters of $A$. We apply Theorem \ref{theorem--vanishing theorem for maps of Tor} to $M=A/(x_1,\dots,x_d)$. We have $$\Tor_i^A(A/(x_1,\dots,x_d), R)\to \Tor_i^A(A/(x_1,\dots,x_d), S)$$ vanishes for all $i\geq 1$. However, we also know that this map is injective because $R\to S$ is pure. Thus we have $\Tor_i^A(A/(x_1,\dots,x_d), R)=H_i(x_1,\dots,x_d, R)=0$ for all $i\geq 1$. This implies $x_1,\dots,x_d$ is a regular sequence on $R$ and hence $R$ is Cohen-Macaulay. Obviously, the complete local domain $R$ is analytically unramified.

We now check the last condition of pseudo-rationality. Let $W\to \Spec R$ be a projective birational map, thus $W\cong \Proj T=\Proj R\oplus Jt\oplus J^2t^2\oplus\cdots$ for some ideal $J\subseteq R$. We now apply the Sancho de Salas exact sequence (\ref{equation--sancho de salas}) to get:
\[ \xymatrix{
[H_{\m+T_{>0}}^d(T)]_0 \ar[r] \ar@{^{(}->}[d]  & H_\m^d(R) \ar[r]\ar[d]^= & H_E^d(W,O_W)  \\
H_{\m+T_{>0}}^d(T) \ar[r] & H_\m^d(R).
}\]

Thus in order to show $H_\m^d(R)\to H_E^d(W,O_W)$ is injective, it suffices to show $H_{\m+T_{>0}}^d(T)\to H_\m^d(R)$ vanishes. We can localize $T$ at the maximal ideal $\m+T_{>0}$, complete, and kill a minimal prime without affecting whether the map vanishes or not. Hence it is enough to show that if $(T,\m) \twoheadrightarrow (R,\m)$ is a surjection such that $T$ is a complete local domain of dimension $d+1$, then $H_\m^d(T)\to H_\m^d(R)$ vanishes. By Cohen's structure theorem there exists $(A,\m_0)\to (T,\m)$ a module-finite extension such that $A$ is a complete regular local ring. We consider the chain of maps
$$A\to T\to R\to S$$
Applying Theorem \ref{theorem--vanishing theorem for maps of Tor} to $A\to T\to S$ and $M=H_{\m_0}^{d+1}(A)$, we know that the composite map
$$\Tor_1^A(H_{\m_0}^{d+1}(A), T)\to \Tor_1^A(H_{\m_0}^{d+1}(A), R)\to \Tor_1^A(H_{\m_0}^{d+1}(A), S)$$
vanishes. Since the \v{C}ech complex on a regular system of parameters gives a flat resolution of $H_{\m_0}^{d+1}(A)$ over $A$, we know that $\Tor_1^A(H_{\m_0}^{d+1}(A), N)\cong H_{\m_0}^{d}(N)$ for every $A$-module $N$. Thus the composite map $$H_\m^d(T)\to H_\m^d(R)\to H_\m^d(S)$$ vanishes. But then $H_\m^d(T)\to H_\m^d(R)$ vanishes because $H_\m^d(R)\to H_\m^d(S)$ is injective since $R\to S$ is pure.
\end{proof}

\begin{remark}
\label{remark--Boutot also follows from Andre and Bhatt}
Corollary \ref{corollary--mixed char Boutot} can be also obtained by combining the main results of \cite{AndreDirectsummandconjecture}, \cite{BhattDirectsummandandDerivedvariant} and using the following argument: the existence of weakly functorial big Cohen-Macaulay algebras for injective ring homomorphisms \cite[Remarque 4.2.1]{AndreDirectsummandconjecture} implies that direct summands of regular rings are Cohen-Macaulay, but we also know they are derived splinters (because this is true for regular rings by \cite[Theorem 1.2]{BhattDirectsummandandDerivedvariant} and it is easy to see that direct summand of derived splinters are still derived splinters). Now the argument of \cite[Lemma 7.5]{KovacsSuperRational} implies that Cohen-Macaulay derived splinters are pseudo-rational.
\end{remark}

\begin{remark}
\label{remark--derived direct summand}
Last we point out that by \cite[Remark 5.12]{MaThevanishingconjectureformapsofTorandderivedsplinters}, Theorem \ref{theorem--vanishing theorem for maps of Tor} gives a new proof of the derived direct summand conjecture \cite[Theorem 6.1]{BhattDirectsummandandDerivedvariant}, that is, if $R$ is a complete regular local ring of mixed characteristic and $\pi$: $X\to \Spec R$ is a proper surjective map, then $R\to \mathbf{R}\pi_*O_X$ splits in the derived category of $R$-modules. Our proof is different from \cite{BhattDirectsummandandDerivedvariant} as it does not use Scholze's vanishing theorem \cite[Proposition 6.14]{ScholzePerfectoidspaces}. In fact, tracing the arguments of \cite[Theorem 5.11 and Remark 5.13]{MaThevanishingconjectureformapsofTorandderivedsplinters}, one can show that our Theorem \ref{theorem--big CM} leads to a stronger result that complete local rings that are pure inside all their big Cohen-Macaulay algebras (e.g., complete regular local rings) are derived splinters.
\end{remark}

\bibliographystyle{skalpha}
\bibliography{CommonBib}

\end{document}